\documentclass[
final
]{dmtcs-episciences}

\usepackage{latexsym}
\usepackage[utf8]{inputenc}
\usepackage[T1]{fontenc}
\usepackage{amsmath}
\usepackage{amssymb}

\usepackage{tikz}
\usetikzlibrary{arrows,shapes,positioning}
\usetikzlibrary{decorations.markings}
\tikzstyle arrowstyle=[draw]
\tikzstyle directed=[postaction={decorate,decoration={markings,
mark=at position 0 with {\node[circle,fill=black,inner sep=0.3ex] {};},
mark=at position 0.5 with {\arrow[scale=1.2,arrowstyle]{stealth};},
mark=at position 1 with {\node[circle,fill=black,inner sep=0.3ex] {};},
}}]

\newtheorem{theorem}{Theorem}

\newtheorem{corollary}[theorem]{Corollary}
\newtheorem{lemma}[theorem]{Lemma}
\newtheorem{observation}[theorem]{Observation}
\newtheorem{conjecture}[theorem]{Conjecture}

\newcommand{\an}{an}
\newcommand{\upword}{upword}
\newcommand{\upwords}{upwords}

\newcommand{\ol}{\overline}

\author{Herman~Z.~Q.~Chen\affiliationmark{1}
  \and Sergey~Kitaev\affiliationmark{2}
  \and Torsten~M\"{u}tze\affiliationmark{3}
  \and Brian Y. Sun\affiliationmark{4}\thanks{The last author was supported by the Scientific Research Program of the Higher Education Institution of Xinjiang Uygur Autonomous Region (No.~XJEDU2016S032) and the Natural Science Foundation of Xinjiang University.}}
\title[On universal partial words]{On universal partial words}
\affiliation{
School of Science, Tianjin Chengjian University, P.R.\ China \\
Department of Computer and Information Sciences, University of Strathclyde, Glasgow, UK \\
Institut f\"{u}r Mathematik, TU Berlin, Germany \\
College of Mathematics and System Science, Xinjiang University, Urumqi, Xinjiang 830046, P.R.\ China}
\keywords{universal word, partial word, De Bruijn graph, Eulerian cycle, Hamiltonian cycle}

\received{2016-11-10}
\accepted{2017-5-16}
\revised{2017-5-8}

\begin{document}
\publicationdetails{19}{2017}{1}{16}{2025}
\maketitle
\begin{abstract}
A {\em universal word} for a finite alphabet $A$ and some integer $n\geq 1$ is a word over $A$ such that every word in $A^n$ appears exactly once as a subword (cyclically or linearly).
It is well-known and easy to prove that universal words exist for any $A$ and $n$.
In this work we initiate the systematic study of universal {\em partial} words.
These are words that in addition to the letters from $A$ may contain an arbitrary number of occurrences of a special `joker' symbol $\Diamond\notin A$, which can be substituted by any symbol from $A$.
For example, $u=0\Diamond 011100$ is a linear partial word for the binary alphabet $A=\{0,1\}$ and for $n=3$ (e.g., the first three letters of $u$ yield the subwords $000$ and $010$).
We present results on the existence and non-existence of linear and cyclic universal partial words in different situations (depending on the number of $\Diamond$s and their positions), including various explicit constructions.
We also provide numerous examples of universal partial words that we found with the help of a computer. \\
\end{abstract}

\section{Introduction}
\label{sec:intro}

De Bruijn sequences are a centuries-old and well-studied topic in combinatorics, and over the years they found widespread use in real-world applications, e.g., in the areas of molecular biology \cite{compeau:11}, computer security \cite{MR653429}, computer vision \cite{pscf:05}, robotics \cite{scheinerman:01} and psychology experiments \cite{sbsh:97}.
More recently, they have also been studied in a more general context by constructing {\em universal cycles} for other fundamental combinatorial structures such as permutations or subsets of a fixed ground set (see e.g.~\cite{MR2638827, MR1197444, MR2925746, MR3193758}).

In the context of words over a finite alphabet $A$, we say that a word $u$ is {\em universal for $A^n$} if $u$ contains every word of length $n\geq 1$ over $A$ exactly once as a subword.
We distinguish \emph{cyclic universal words} and \emph{linear universal words}.
In the cyclic setting, we view $u$ as a cyclic word and consider all subwords of length $n$ cyclically across the boundaries of $u$.
In the linear setting, on the other hand, we view $u$ as a linear word and only consider subwords that start and end within the index range of letters of $u$.
From this definition it follows that the length of a cyclic or linear universal word must be $|A|^n$ or $|A|^n+n-1$, respectively.
For example, for the binary alphabet $A=\{0,1\}$ and for $n=3$, $u=0001011100$ is a linear universal word for $A^3$.
Observe that a cyclic universal word for $A^n$ can be easily transformed into a linear universal word for $A^n$, so existence results in the cyclic setting imply existence results for the linear setting.
Note also that reversing a universal word, or permuting the letters of the alphabet yields a universal word again.
The following classical result is the starting point for our work (see \cite{de-bruijn:46,MR0142475,MR0276108}).

\begin{theorem}
\label{thm:universal}
For any finite alphabet $A$ and any $n\geq 1$, there exists a cyclic universal word for $A^n$.
\end{theorem}

The standard proof of Theorem~\ref{thm:universal} is really beautiful and concise, using the De Bruijn graph, its line graph and Eulerian cycles (see \cite{MR1197444} and Section~\ref{sec:prelim} below).

\subsection{Universal partial words}

In this paper we consider the universality of {\em partial words}, which are words that in addition to letters from $A$ may contain any number of occurrences of an additional special symbol $\Diamond\notin A$. The idea is that every occurrence of $\Diamond$ can be substituted by any symbol from $A$, so we can think of $\Diamond$ as a `joker' or `wildcard' symbol.
Formally, we define $A_\Diamond:=A\cup\{\Diamond\}$ and we say that a word $v=v_1v_2\cdots v_n\in A^n$ appears as a {\em factor} in a word $u=u_1u_2\cdots u_m\in A_\Diamond^m$ if there is an integer $i$ such that $u_{i+j}=\Diamond$ or $u_{i+j}=v_j$ for all $j=1,2,\ldots,n$.
In the cyclic setting we consider the indices of $u$ in this definition modulo $m$.
For example, in the linear setting and for the ternary alphabet $A=\{0,1,2\}$, the word $v=021$ occurs twice as a factor in $u=120\Diamond 120021$ because of the subwords $0\Diamond 1$ and $021$ of $u$, whereas $v$ does not appear as a factor in $u'=12\Diamond 11\Diamond$.

Partial words were introduced in \cite{MR1687780}, and they too have real-world applications (see \cite{blanchet-sadri:08} and references therein).
In combinatorics, partial words appear in the context of primitive words \cite{blanchet-sadri:05}, of (un)avoidability of sets of partial words \cite{MR2511569,MR2779632}, and also in the study of the number of squares~\cite{MR2456602} and overlap-freeness~\cite{MR2492032} in (infinite) partial words.
The concept of partial words has been extended to pattern-avoiding permutations in \cite{MR2770130}.

The notion of universality given above extends straightforwardly to partial words, and we refer to a universal partial word as \an{} {\em \upword{}} for short.
Again we distinguish cyclic \upwords{} and linear \upwords{}.
The simplest example for a linear \upword{} for $A^n$ is $\Diamond^n:=\Diamond\Diamond\cdots\Diamond$, the word consisting of $n$ many $\Diamond$s, which we call {\em trivial}.
Let us consider a few more interesting examples of linear \upwords{} over the binary alphabet $A=\{0,1\}$.
We have that $\Diamond\Diamond 0111$ is a linear \upword{} for $A^3$, whereas $\Diamond\Diamond 01110$ is {\em not} a linear \upword{} for $A^3$, because replacing the first two letters $\Diamond\Diamond$ by $11$ yields the same factor $110$ as the last three letters.
Similarly, $0\Diamond 1$ is {\em not} a linear \upword{} for $A^2$ because the word $10\in A^2$ does not appear as a factor (and the word $01\in A^2$ appears twice as a factor).

\subsection{Our results}

In this work we initiate the systematic study of universal partial words.
It turns out that these words are rather shy animals, unlike their ordinary counterparts (universal words without `joker' symbols).
That is, in stark contrast to Theorem~\ref{thm:universal}, there are no general existence results on \upwords{}, but also many non-existence results.
The borderline between these two cases seems rather complicated, which makes the subject even more interesting (this is true also for non-binary alphabets, as the constructions of the follow-up paper \cite{kirsch:16} indicate).
In addition to the size of the alphabet $A$ and the length $n$ of the factors, we also consider the number of $\Diamond$s and their positions in \an{} \upword{} as problem parameters.

\begin{table}[ht!]
\begin{center}
\begin{tabular}{c|c|l}
$n$ & $k$ & \\ \hline
1   & 1 & $\Diamond$ \\ \hline
2   & 1 & $\Diamond 011$ (Thm.~\ref{thm:diamond-at-pos-1}, Thm.~\ref{thm:nm1-diamonds}) \\
    & 2 & --- (Thm.~\ref{thm:diamond-at-pos-n}) \\ \hline
3   & 1 & $\Diamond 00111010$ (Thm.~\ref{thm:diamond-at-pos-1}) \\
    & 2 & $0\Diamond 011100$ (Thm.~\ref{thm:diamond-at-pos-k}) \\
    & 3 & --- (Thm.~\ref{thm:diamond-at-pos-n}) \\
    & 4 & --- (Thm.~\ref{thm:diamond-at-pos-57}) \\ \hline
4   & 1 & $\Diamond 00011110100101100$ (Thm.~\ref{thm:diamond-at-pos-1}) \\
    & 2 & $0\Diamond 010011011110000$ (Thm.~\ref{thm:diamond-at-pos-k}) \\
    & 3 & $01\Diamond 0111100001010$ (Thm.~\ref{thm:diamond-at-pos-k}) \\
    & 4 & --- (Thm.~\ref{thm:diamond-at-pos-n}) \\
    & 5 & --- (Thm.~\ref{thm:diamond-at-pos-57}) \\
    & 6 & $01100\Diamond 011110100$ \\
    & 7 & --- (Thm.~\ref{thm:diamond-at-pos-57}) \\
    & 8 & $0011110\Diamond 0010110$ \\ \hline
5   & 1 & $\Diamond 0000111110111001100010110101001000$ (Thm.~\ref{thm:diamond-at-pos-1}) \\
    & 2 & $0\Diamond 01011000001101001110111110010001$ (Thm.~\ref{thm:diamond-at-pos-k}) \\
    & 3 & $01\Diamond 011000001000111001010111110100$ (Thm.~\ref{thm:diamond-at-pos-k}) \\
    & 4 & $011\Diamond 0111110000010100100011010110$ (Thm.~\ref{thm:diamond-at-pos-k}) \\
    & 5 & --- (Thm.~\ref{thm:diamond-at-pos-n}) \\
    & 6 & $00101\Diamond 0010011101111100000110101$ \\
    & 7 & $010011\Diamond 010000010101101111100011$ \\
    & 8 & $0100110\Diamond 01000001110010111110110$ \\
    & 9 & $01110010\Diamond 0111110110100110000010$ \\
    & 10 & $010011011\Diamond 010001111100000101011$ \\
    & 11 & $0101000001\Diamond 01011111001110110001$ \\
    & 12 & $01010000011\Diamond 0101101111100010011$ \\
    & 13 & $001001101011\Diamond 001010000011111011$ \\
    & 14 & $0011101111100\Diamond 00110100010101100$ \\
    & 15 & $01010000010011\Diamond 0101101111100011$ \\
    & 16 & $001000001101011\Diamond 001010011111011$
\end{tabular}
\caption{Examples of linear \upwords{} for $A^n$, $A=\{0,1\}$, with a single $\Diamond$ at position $k$ from the beginning or end for $n=1,2,3,4,5$ and all possible values of $k$ (\upwords{} where the $\Diamond$ is closer to the end of the word than to the beginning can be obtained by reversal). A dash indicates that no such \upword{} exists.}
\label{tab:upwords1}
\end{center}
\end{table}

\begin{table}[ht!]
\begin{center}
\begin{tabular}{l|l}
$n$   & \\\hline
2     & $\Diamond\Diamond$ (Cor.~\ref{cor:dia-dia}) \\ \hline
3     & $\Diamond\Diamond 0111$ (Cor.~\ref{cor:dia-dia}, Thm.~\ref{thm:nm1-diamonds}) \\
      & $\Diamond 001011\Diamond$ \\ \hline
4     & $\Diamond 00011\Diamond 1001011$ (Thm.~\ref{thm:two-diamonds}) \\
      & $\Diamond 0001011\Diamond 10011$ \\
      & $001\Diamond 110\Diamond 001$ \\ \hline
5     & $\Diamond 0100\Diamond 101011000001110111110010$ \\
      & $\Diamond 0000111\Diamond 100010010101100110111$ (Thm.~\ref{thm:two-diamonds}) \\
      & $\Diamond 00001001\Diamond 10001101011111011001$ \\
      & $\Diamond 0000100111\Diamond 100011001010110111$ \\
      & $\Diamond 00001010111\Diamond 10001101100100111$ \\
      & $0\Diamond 0011\Diamond 0100010101101111100000$ \\
      & $0\Diamond 010110\Diamond 00011101111100100110$ \\
      & $0\Diamond 0101110\Diamond 0001101100100111110$ \\
      & $0\Diamond 010111110\Diamond 00011011001001110$ \\
      & $0\Diamond 0101101110\Diamond 0001100100111110$ \\
      & $00\Diamond 0011\Diamond 00101011011111010000$ \\
      & $01\Diamond 01100101110\Diamond 0100000111110$ \\
      & $01\Diamond 0110010111110\Diamond 01000001110$ \\
      & $01\Diamond 0100000101011000111110\Diamond 110$ \\
      & $001\Diamond 0101\Diamond 001110111110000010$ \\
      & $011\Diamond 011010010\Diamond 0111110000010$ \\
      & $011\Diamond 0110101001000\Diamond 011111000$ \\
      & $011\Diamond 0111110001101010000010\Diamond 10$ \\
      & $011\Diamond 011010000011111000100101\Diamond 1$ \\
      & $01001\Diamond 1110\Diamond 010000011011001$
\end{tabular}
\caption{Examples of linear \upwords{} for $A^n$, $A=\{0,1\}$, with two $\Diamond$s for $n=2,3,4,5$.}
\label{tab:upwords2}
\end{center}
\end{table}

We first focus on linear \upwords{}.
For linear \upwords{} containing a {\em single} $\Diamond$, we have the following results:
For non-binary alphabets $A$ (i.e., $|A|\geq 3$) and $n\geq 2$, there is {\em no} linear \upword{} for $A^n$ with a single $\Diamond$ at all (Theorem~\ref{thm:alpha3} below).
For the binary alphabet $A=\{0,1\}$, the situation is more interesting (see Table~\ref{tab:upwords1}):
Denoting by $k$ the position of the $\Diamond$, we have that for $n\geq 2$, there is {\em no} linear \upword{} for $A^n$ with $k=n$ (Theorem~\ref{thm:diamond-at-pos-n}), and there are {\em no} linear \upwords{} in the following three cases: $n=3$ and $k=4$, and $n=4$ and $k\in\{5,7\}$ (Theorem~\ref{thm:diamond-at-pos-57}).
We conjecture that these are the only non-existence cases for a binary alphabet (Conjecture~\ref{conj:single-diamond}).
To support this conjecture, we performed a computer-assisted search and indeed found linear \upwords{} for all values of $2\leq n\leq 13$ and all possible values of $k$ other than the ones excluded by the beforementioned results.
Some of these examples are listed in Table~\ref{tab:upwords1}, and the remaining ones are available on the third author's website \cite{www}.
We also prove the special cases $k=1$ and $k\in\{2,3,\ldots,n-1\}$ of our conjecture (Theorems~\ref{thm:diamond-at-pos-1} and \ref{thm:diamond-at-pos-k}, respectively).

For linear \upwords{} containing {\em two} $\Diamond$s we have the following results:
First of all, Table~\ref{tab:upwords2} shows examples of linear \upwords{} with two $\Diamond$s for the binary alphabet $A=\{0,1\}$ for $n=2,3,4,5$.
We establish a sufficient condition for non-existence of binary linear \upwords{} with two $\Diamond$s (Theorem~\ref{thm:two-diamonds-non}), which in particular shows that a $(1-o(1))$-fraction of all ways of placing two $\Diamond$s among the $N=\Theta(2^n)$ positions does not yield a valid \upword{}.
Moreover, we conclude that there are only two binary linear \upwords{} where the two $\Diamond$s are adjacent (Corollary~\ref{cor:dia-dia}), namely $\Diamond\Diamond$ for $n=2$ and $\Diamond\Diamond 0111$ for $n=3$ (see Table~\ref{tab:upwords2}).
We also construct an infinite family of binary linear \upwords{} with two $\Diamond$s (Theorem~\ref{thm:two-diamonds}).

Let us now discuss cyclic \upwords{}.
Note that the trivial solution $\Diamond^n$ is a cyclic \upword{} only for $n=1$.
For the cyclic setting we have the following rather general non-existence result:
If $\gcd(|A|,n)=1$, then there is no cyclic \upword{} for $A^n$ (Corollary~\ref{cor:cyclic-div}).
In particular, for a binary alphabet $|A|=2$ and odd $n$, there is no cyclic \upword{} for $A^n$.
In fact, we know only of a single cyclic \upword{} for the binary alphabet $A=\{0,1\}$ and any $n\geq 2$, namely $\Diamond 001\Diamond 110$ for $n=4$ (up to cyclic shifts, reversal and letter permutations).

\subsection{Outline of this paper}

This paper is organized as follows.
In Section~\ref{sec:prelim} we introduce some notation and collect basic observations that are used throughout the rest of the paper.
In Sections~\ref{sec:single-diamond} and \ref{sec:two-diamonds} we prove our results on linear \upwords{} containing a single or two $\Diamond$s, respectively.
Section~\ref{sec:cyclic} contains the proofs on cyclic \upwords{}.
Finally, Section~\ref{sec:outlook} discusses possible directions for further research, including some extensions of our results to non-binary alphabets.

\section{Preliminaries}
\label{sec:prelim}

For the rest of this paper, we assume w.l.o.g.\ that the alphabet is $A=\{0,1,\ldots,\alpha-1\}$, so $\alpha\geq 2$ denotes the size of the alphabet.
We often consider the special case $\alpha=2$ of the binary alphabet, and then for $x\in\{0,1\}$ we write $\ol{x}$ for the complement of $x$.
Moreover, for any word $u$, we let $|u|$ denote its length.
As we mentioned before, reversing a universal word and/or permuting the letters of the alphabet again yields a universal word.
We can thus assume w.l.o.g.\ that in \an{} \upword{} $u$ the letters of $A$ appear in increasing order from left to right, i.e., the first occurence of symbol $i$ is before the first occurence of symbol $j$ whenever $i<j$.
Moreover, if $u$ can be factored as $u=xyz$, where $x$ and $z$ do not contain any $\Diamond$s, then we can assume that $|x|\leq |z|$.

One standard approach to prove the existence of universal words is to define a suitable graph and to search for a Hamiltonian path/cycle in this graph (another more algebraic approach uses irreducible polynomials).
Specifically, the {\em De Bruijn graph} $G_A^n$ has as vertices all elements from $A^n$ (all words of length $n$ over $A$), and a directed edge from a vertex $u$ to a vertex $v$ whenever the last $n-1$ letters of $u$ are the same as the first $n-1$ letters of $v$.
We call such an edge $(u,v)$ an {\em $x$-edge}, if the last letter of $v$ equals $x$.
Figure~\ref{fig:bruijn} (a) and (b) shows the graph $G_A^n$, $A=\{0,1\}$, for $n=2$ and $n=3$, respectively.
Clearly, a linear universal word for $A^n$ corresponds to a Hamiltonian path in $G_A^n$, and a cyclic universal word to a Hamiltonian cycle in this graph.
Observe furthermore that $G_A^n$ is the line graph of $G_A^{n-1}$.
Recall that the {\em line graph} $L(G)$ of a directed graph $G$ is the directed graph that has a vertex for every edge of $G$, and a directed edge from $e$ to $e'$ if in $G$ the end vertex of $e$ equals the starting vertex of $e'$.
Therefore, the problem of finding a Hamiltonian path/cycle in $G_A^n$ is equivalent to finding an Eulerian path/cycle in $G_A^{n-1}$.
The existence of an Eulerian path/cycle follows from the fact that the De Bruijn graph is connected and that each vertex has in- and out-degree $\alpha$ (this is one of Euler's famous theorems \cite{euler}, see also \cite[Theorem 1.6.3]{bang-jensen-gutin:08}).
This proves Theorem~\ref{thm:universal}.
In fact, this existence proof can be easily turned into an algorithm to actually find (many) universal words (using Hierholzer's algorithm \cite{MR1509807} or Fleury's algorithm \cite{fleury}).

\begin{figure}[ht!]
\begin{center}
\mbox{
\hspace{-10mm}
\begin{tikzpicture}[scale=1.1]
\begin{scope}[shift={(0,2.0)}]
% grey path in the background
\draw[line width=4pt,color=lightgray]
% upper circle
(0.25,2.1) to [out=30,in=-90] (0.4,2.4)
to [out=90,in=0] (0.,2.8)
to [out=180,in=90] (-0.4,2.4)
to [out=-90,in=180] (0,2)--++(1,-1)--++(-0.8,-0.8)--++(0,-0.25)
% lower circle
to [out=0,in=90] (0.4,-0.4)
to [out=-90,in=0] (0,-0.8)
to [out=180,in=-90] (-0.4,-0.4)
to [out=90,in=200] (-0.2,-0.05)--++(0,0.25)--++(-0.8,0.8)--++(0.8,0.8);
% remaining middle part
\draw[line width=4pt,color=lightgray] (0.55,0.83) to [out=190,in=-10] (-0.55,0.83)--++(0,0.27) to [out=15,in=165] (1,1);

\node[right] at (0,0){$11$};
\node[left] at (-1,1){$10$};
\node[left] at (0,2){$00$};
\node[right] at (1,1){$01$};

\draw [thick,dashed,directed] (0,0)--++(-1,1);
\draw [thick,dashed,directed] (-1,1)--++(1,1);
\draw [thick,dashed,directed] (0,2)--++(1,-1);
\draw [thick,dashed,directed] (1,1)--++(-1,-1);
\draw [thick,dashed,directed] (-1,1) .. controls (0.5,1.3) and (-0.5,1.3) .. (1,1);
\draw [thick,dashed,directed] (1,1) .. controls (0.5,0.7) and (-0.5,0.7) .. (-1,1);
\draw [thick,dashed,directed] (0,2) arc (-90:270:0.4);
\draw [thick,dashed,directed] (0,0) arc (90:-270:0.4);
\node at (-1,3) {$G_A^2$};
\node at (0,-3) {(a)};
\end{scope}
\path (0,0);
\end{tikzpicture}
\begin{tikzpicture}[scale=1.1]
\node[right] at (0,0.5){$010$};
\node[left] at (-1,1){$100$};
\node[right] at (0,2){$000$};
\node[right] at (1,1){$001$};
\node[right] at (0,-0.5){$101$};
\node[left] at (-1,-1){$110$};
\node[right] at (0,-2){$111$};
\node[right] at (1,-1){$011$};

\draw [thick,dashed,directed] (-1,1)--++(2,0);
\draw [thick,dashed,directed] (-1,1)--++(1,1);
\draw [thick,dashed,directed] (1,1)--++(-1,-0.5);
\draw [thick,dashed,directed] (1,-1)--++(-2,0);
\draw [thick,dashed,directed] (-1,-1)--++(1,0.5);
\draw [thick,dashed,directed] (0,0.5)--++(-1,0.5);

\draw [thick,dashed,directed] (0,-0.5) .. controls (-0.2,0.2) and (-0.2,-0.2) .. (0,0.5);
\draw [thick,dashed,directed] (0,2) arc (-90:270:0.4);
\draw [thick,dashed,directed] (0,-2) arc (90:-270:0.4);

\draw [very thick,directed] (0,2)--++(1,-1);
\draw [very thick,directed] (1,1)--++(0,-2);
\draw [very thick,directed] (0,0.5) .. controls (0.2,0.2) and (0.2,-0.2) .. (0,-0.5);
\draw [very thick,directed] (0,-0.5)--++(1,-0.5);
\draw [very thick,directed] (-1,-1)--++(0,2);
\draw [very thick,directed] (1,-1)--++(-1,-1);
\draw [very thick,directed] (0,-2)--++(-1,1);

\node at (-0.5,3.3) {$G_A^3=L(G_A^2)$};
\node at (-1,-3) {(b)};
\end{tikzpicture}
\begin{tikzpicture}[scale=1.1]
\begin{scope}[shift={(0,2.7)}]
\draw [thick,directed] (-1.5,0.8)--++(1,0);
\draw [thick,directed] (-0.5,0.8)--++(0.5,-0.8);
\draw [thick,directed] (0,0)--++(1,0);
\draw [thick,directed] (1,0)--++(1,0);
\draw [thick,directed] (2,0)--++(1,0);
\draw [thick,directed] (-1.5,-0.8)--++(1,0);
\draw [thick,directed] (-0.5,-0.8)--++(0.5,0.8);

\draw [dashed] (-1.5,0) ellipse (0.3 and 1.05);
\draw [dashed] (-0.5,0) ellipse (0.3 and 1.05);

\node at (-1.5,1.3) {$S(u,1,n)$};
\node at (0,1.3) {$S(u,2,n)$};
\node at (0.5,0.3) {$S(u,3,n)$};
\node at (1,-0.3) {$S(u,4,n)$};
\node at (2.1,0.3) {$S(u,5,n)$};
\node at (3,-0.3) {$S(u,6,n)$};
\node at (0.5,-2) {$H(u,n)$, $u=0\Diamond 011100$};
\node at (0,-3) {(c)};
\end{scope}
\path (0,0);
\end{tikzpicture}
}
\caption{The De Bruijn graphs $G_A^2$ (a) and $G_A^3=L(G_A^2)$ (b) for $A=\{0,1\}$ with a spanning subgraph $H(u,n)$ of $G_A^3$ for the linear \upword{} $u=0\Diamond 011100$ for $A^3$ ($H(u,n)$ is shown by solid edges). Part (c) if the figure shows a schematic drawing of the graph $H(u,n)$. $H(u,n)$ is the line graph of the highlighted sequences of edges in $G_A^2$.}
\label{fig:bruijn}
\end{center}
\end{figure}
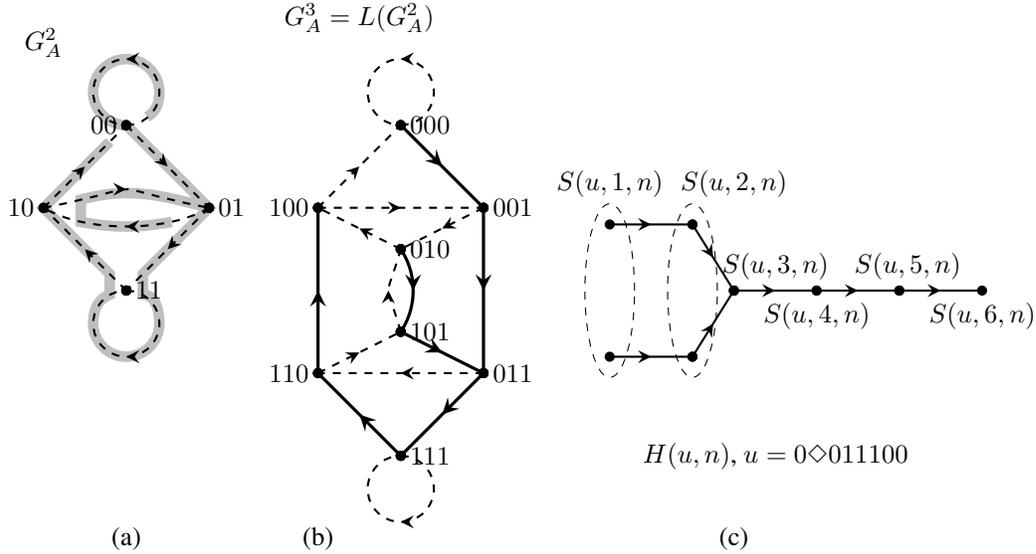

We now discuss how this standard approach of proving the existence of universal words can be extended to universal partial words.
Specifically, we collect a few simple but powerful observations that will be used in our proofs later on.

For any vertex $v$ of $G_A^n$, we let $\Gamma^+(v)$ and $\Gamma^-(v)$ denote the sets of out-neighbours and in-neighbours of $v$, respectively (both are sets of vertices of $G_A^n$).
As we mentioned before, we clearly have $|\Gamma^+(v)|=|\Gamma^-(v)|=\alpha$.

\begin{observation}
\label{obs:common-neighbors}
For any vertex $v=v_1v_2\cdots v_n$ of $G_A^n$ and its set of out-neighbours $\Gamma^+(v)$, there are $\alpha-1$ vertices different from $v$ with the same set of out-neighbours $\Gamma^+(v)$, given by $xv_2v_3\cdots v_n$, where $x\in A\setminus \{v_1\}$.
For any vertex $v=v_1v_2\cdots v_n$ of $G_A^n$ and its set of in-neighbours $\Gamma^-(v)$, there are $\alpha-1$ vertices different from $v$ with the same set of in-neighbours $\Gamma^-(v)$, given by $v_1v_2\cdots v_{n-1}x$, where $x\in A\setminus \{v_n\}$.
\end{observation}

For any linear \upword{} $u$ for $A^n$, we define a spanning subgraph $H(u,n)$ of the De Bruijn graph $G_A^n$ as follows, see Figure~\ref{fig:bruijn} (c):
For any $i=1,2,\ldots,N-n+1$, we let $S(u,i,n)$ denote the set of all words that are obtained from the subword of $u$ of length $n$ starting at position $i$ by replacing any occurences of $\Diamond$ by a letter from the alphabet $A$.
Clearly, if there are $d$ many $\Diamond$s in this subword, then there are $\alpha^d$ different possibilities for substitution, so we have $|S(u,i,n)|=\alpha^d$.
Note that the sets $S(u,i,n)$ form a partition of the vertex set of $G_A^n$ (and $H(u,n)$).
The directed edges of $H(u,n)$ are given by all the edges of $G_A^n$ induced between every pair of consecutive sets $S(u,i,n)$ and $S(u,i+1,n)$ for $i=1,2,\ldots,N-n$.
For example, for the linear \upword{} $u=0\Diamond 011100$ over the binary alphabet $A=\{0,1\}$ for $n=3$ we have $S(u,1,n)=\{000,010\}$, $S(u,2,n)=\{001,101\}$, $S(u,3,n)=\{011\}$, $S(u,4,n)=\{111\}$, $S(u,5,n)=\{110\}$ and $S(u,6,n)=\{100\}$, and the spanning subgraph $H(u,n)$ of $G_A^3$ is shown in Figure~\ref{fig:bruijn} (c).
To give another example with the same $A$ and $n$, for the linear \upword{} $u=\Diamond\Diamond 0111$ we have $S(u,1,n)=\{000,010,100,110\}$, $S(u,2,n)=\{001,101\}$, $S(u,3,n)=\{011\}$, $S(u,4,n)=\{111\}$, and then $H(u,n)$ is a binary tree of depth 2 with an additional edge emanating from the root.

The following observation follows straightforwardly from these definitions.

\begin{observation}
\label{obs:Hu-degrees}
Let $u=u_1u_2\cdots u_N$ be a linear \upword{} for $A^n$.
A vertex in $S(u,i,n)$, $i=1,2,\ldots,N-n$, has out-degree $1$ in $H(u,n)$ if $u_{i+n}\in A$, and out-degree $\alpha$ if $u_{i+n}=\Diamond$.
A vertex in $S(u,i,n)$, $i=2,3,\ldots,N-n+1$, has in-degree $1$ in $H(u,n)$ if $u_{i-1}\in A$, and in-degree $\alpha$ if $u_{i-1}=\Diamond$.
The vertices in $S(u,1,n)$ have in-degree $0$, and the vertices in $S(u,N-n+1,n)$ have out-degree $0$.
\end{observation}

By this last observation, the graph $H(u,n)$ is determined only by the positions of the $\Diamond$s in $u$.
Intuitively, the $\Diamond$s lead to branching in the graph $H(u,n)$ due to the different possibilities of substituting symbols from $A$.
In particular, if $u$ has no $\Diamond$s, then $H(u,n)$ is just a spanning path of $G_A^n$ (i.e., a Hamiltonian path, so we are back in the setting of Theorem~\ref{thm:universal}).
So when searching for a linear universal partial word $u$ with a particular number of $\Diamond$s at certain positions, we essentially search for a copy of the spanning subgraph $H(u,n)$ in $G_A^n$.
We will exploit this idea both in our existence and non-existence proofs.
For the constructions it is particularly useful (and for our computer-searches it is computationally much more efficient) to not search for a copy of $H(u,n)$ in $G_A^n$ directly, but to rather search for the corresponding sequences of edges in $G_A^{n-1}$, which can be seen as generalizations of Eulerian paths that were used before in the proof of Theorem~\ref{thm:universal} (see Figure~\ref{fig:bruijn} (a)).
For example, to search for a linear \upword{} $u$ with a single $\Diamond$ at position $k\in\{1,2,\ldots,n-1\}$, we can prescribe the first $k-1$ letters and the $n$ letters after the $\Diamond$ (with a particular choice of symbols from $A$, or by iterating over all possible choices), and search for an Eulerian path in the subgraph of $G_A^{n-1}$ that remains when deleting from it all edges that correspond to the prescribed prefix of $u$ (see the proofs of Theorems~\ref{thm:diamond-at-pos-1} and \ref{thm:diamond-at-pos-k} below).
This idea can be generalized straightforwardly to search for \upwords{} with other $\Diamond$ patterns (see for example the proof of Theorem~\ref{thm:two-diamonds} below).

The next lemma will be used repeatedly in our proofs (both for existence and non-existence of \upwords{}).
The proof uses the previous two graph-theoretical observations to derive dependencies between letters of \an{} \upword.

\begin{lemma}
\label{lem:constraint}
Let $u=u_1u_2\cdots u_N$ be a linear \upword{} for $A^n$, $A=\{0,1,\ldots,\alpha-1\}$, $n\geq 2$, such that $u_k=\Diamond$ and $u_{k+n}\neq \Diamond$ (we require $k+n\leq N$).
Then for all $i=1,2,\ldots,n-1$ we have that if $u_i\neq \Diamond$, then $u_{k+i}=u_i$.
Moreover, we have that if $u_n\neq \Diamond$, then $\alpha=2$ and $u_{k+n}=\ol{u_n}$.
\end{lemma}

\begin{figure}
\begin{center}
\begin{tikzpicture}[scale=1.1]
\draw [directed,thick](0.6,0.8)--++(1,0);
\draw [directed,thick](1.6,0.8)--++(1,0);
\draw [directed,thick](2.6,0.8)--++(1,0);
\draw [directed,thick](3.6,0.8)--++(1,0);
\draw [directed,thick](4.6,0.8)--++(0.6,-0.8);

\draw [directed,thick](0.6,-0.8)--++(1,0);
\draw [directed,thick](1.6,-0.8)--++(1,0);
\draw [directed,thick](2.6,-0.8)--++(1,0);
\draw [directed,thick](3.6,-0.8)--++(1,0);
\draw [directed,thick](4.6,-0.8)--++(0.6,0.8);

\draw [directed,thick](5.2,0)--++(1,0);
\draw [directed,thick](6.2,0)--++(1,0);
\draw [directed,thick](7.2,0)--++(1,0);

\draw [directed,thick,dashed] (4.6,0.8) .. controls (2.6,1.5) .. (0.6,0.8);
\draw [directed,thick,dashed](4.6,-0.8)--++(-4,1.6);

\node[right] at (0.3,0.5) {$v_x$};
\node[right] at (-1,0.25) {$v_x=v_1\cdots v_{n-1}x$, $x\in A\setminus\{v_n\}$};
\node[right] at (-0.2,-1.2) {$S(u,1,n)$};
\node[right] at (3.9,-1.2) {$S(u,k,n)$};
\node[right] at (5,-0.5) {$S(u,k+1,n)$};
\node[right] at (5.1,0.25) {$v$};
\node[right] at (5.5,0.4) {$v=v_1v_2\cdots v_n$};
\node[right] at (4.5,1.2) {$H(u,n)$ (solid edges)};
\node[right] at (8.4,0) {$\ldots$};
\end{tikzpicture}
\caption{Illustration of the proof of Lemma~\ref{lem:constraint}.}
\label{fig:constraint}
\end{center}
\end{figure}
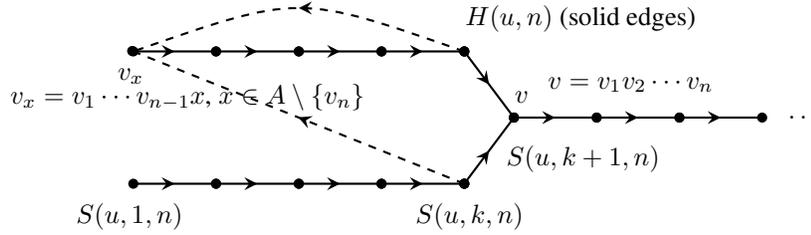

\begin{proof}
By Observation~\ref{obs:Hu-degrees}, each vertex in the set $S(u,k+1,n)$ has in-degree $\alpha$ in $H(u,n)$, and each vertex in $S(u,k,n)$ has out-degree 1.
By Observation~\ref{obs:common-neighbors}, for each $v=v_1v_2\cdots v_n\in S(u,k+1,n)$ there are $\alpha-1$ other vertices (different from the ones in $S(u,k+1,n)$) in $G_A^n$ with the same set $\Gamma^-(v)$ of $\alpha$ many in-neighbors, namely $v_x:=v_1\cdots v_{n-1}x$, where $x\in A\setminus \{v_n\}$ (see Figure~\ref{fig:constraint}).
As the in-degree of every vertex of $G_A^n$ is exactly $\alpha$, and in $H(u,n)$ all vertices except the ones in $S(u,1,n)$ already have in-degree at least 1, it follows that each of the vertices $v_x$ must be equal to one of the vertices in $S(u,1,n)$.
It follows that if $u_i\neq \Diamond$ then $u_{k+i}\neq \Diamond$ and $u_i=v_i=u_{k+i}$ for all $i=1,2,\ldots,n-1$.
Moreover, if $u_n\neq\Diamond$ and $\alpha\geq 3$, then there are at least two vertices $v_x$, $x\in A\setminus\{v_n\}$, ending with different symbols $x$, each of which must be equal to one of the vertices in $S(u,1,n)$, which is impossible because all words in this set end with the same symbol $u_n$.
It follows that if $u_n\neq\Diamond$ then we must have $\alpha=2$ and $u_n=x\neq v_n=u_{k+n}$, so $u_{k+n}=\ol{u_n}$.
\end{proof}

\section{Linear \upwords{} with a single \texorpdfstring{$\Diamond$}{diamond}}
\label{sec:single-diamond}

\subsection{Non-existence results}

Our first result completely excludes the existence of linear \upwords{} with a single $\Diamond$ for non-binary alphabets.

\begin{theorem}
\label{thm:alpha3}
For $A=\{0,1,\ldots,\alpha-1\}$, $\alpha\geq 3$, and any $n\geq 2$, there is no linear \upword{} for $A^n$ with a single $\Diamond$.
\end{theorem}

\begin{proof}
Suppose that such \an{} \upword{} $u=u_1u_2\cdots u_{k-1}\Diamond u_{k+1}\cdots u_N$ exists.
We claim that the $\Diamond$ in $u$ is preceded or followed by at least $n$ symbols from $A$.
If not, then $u$ would have at most $\alpha n$ different factors, which is strictly less than $\alpha^n$ for $\alpha\geq 3$ and $n\geq 2$.
So we assume w.l.o.g.\ that the $\Diamond$ in $u$ is followed by at least $n$ symbols from $A$, i.e., $k+n\leq N$.
By Lemma~\ref{lem:constraint} we have $u_i=\Diamond$ or $u_{k+i}=u_i$ for all $i=1,2,\ldots,n-1$ and $u_n=\Diamond$, which implies $k=n$ and therefore $u_{n+i}=u_i$ for all $i=1,\ldots,n-1$.
But this means that the word $v:=u_{n+1}\cdots u_{2n}\in A^n$ appears twice as a factor in $u$ starting at positions 1 and $n+1$ (in other words, the vertex $v\in S(u,n+1,n)$ is identical to a vertex from $S(u,1,n)$ in $H(u,n)$), a contradiction.
\end{proof}

Our next result excludes several cases with a single $\Diamond$ for a binary alphabet.

\begin{theorem}
\label{thm:diamond-at-pos-n}
For $A=\{0,1\}$ and any $n\geq 2$, there is no linear \upword{} for $A^n$ with a single $\Diamond$ at position $n$ from the beginning or end.
\end{theorem}

\begin{proof}
We first consider the case $n=2$.
Suppose that there is \an{} \upword{} $u=u_1\Diamond u_3$ for $A^n$.
Assuming w.l.o.g.\ that $u_1=0$, we must have $u_3=1$, otherwise the word $00$ would appear twice as a factor.
But then the word $10$ does not appear as a factor in $u=0\Diamond 1$, while 01 appears twice, a contradiction.

For the rest of the proof we assume that $n\geq 3$.
Suppose there was \an{} \upword{} $u=u_1u_2\cdots u_{n-1}\Diamond u_{n+1}\cdots u_N$ with $N=2^n-1$.
Note that $N-n\geq n$, or equivalently $2^n\geq 2n+1$, holds by our assumption $n\geq 3$, so the $\Diamond$ in $u$ is followed by at least $n$ more symbols from $A$.
Applying Lemma~\ref{lem:constraint} yields that $u_{n+i}=u_i$ for all $i=1,\ldots,n-1$, which means that the word $v:=u_{n+1}\cdots u_{2n}\in A^n$ appears twice as a factor in $u$ starting at positions 1 and $n+1$, a contradiction.
\end{proof}

In contrast to Theorem~\ref{thm:alpha3}, for a binary alphabet we can only exclude the following three more (small) cases in addition to the cases excluded by Theorem~\ref{thm:diamond-at-pos-n} (all the exceptions are marked in Table~\ref{tab:upwords1}).

\begin{theorem}
\label{thm:diamond-at-pos-57}
For $A=\{0,1\}$, there is no linear \upword{} for $A^n$ with a single $\Diamond$ at position $k$ from the beginning or end in the following three cases: $n=3$ and $k=4$, and $n=4$ and $k\in\{5,7\}$.
\end{theorem}

\begin{proof}
Suppose that there is \an{} \upword{} $u=u_1u_2u_3\Diamond u_5u_6u_7$ for the case $n=3$.
Applying Lemma~\ref{lem:constraint} twice to $u$ and its reverse we obtain that $u_5u_6u_7=u_1u_2\ol{u_3}$ and $u_1u_2u_3=\ol{u_5}u_6u_7$, a contradiction.

To prove the second case suppose that there is \an{} \upword{} of the form $u=u_1u_2u_3u_4\Diamond u_6\cdots u_{15}$ for $n=4$.
Applying Lemma~\ref{lem:constraint} twice to $u$ and its reverse we obtain that $u$ has the form $u=u_1u_2u_3u_4\Diamond u_1u_2u_3\ol{u_4}u_{10}u_{11}\ol{u_1}u_2u_3u_4$.
We assume w.l.o.g.\ that $u_1=0$.
The word $z:=0000$ must appear somewhere as a factor in $u$, and since $u_{12}=\ol{u_1}=1$, the only possible starting positions for $z$ in $u$ are $1,2,\ldots,8$.
However, the starting positions $1,2,5,6,7$ can be excluded immediately, as they would cause $z$ to appear twice as a factor in $u$.
On the other hand, if $z$ starts at positions 3, 4 or 8, then the neighboring letters must both be 1, causing 0101, 1010 or 1101, respectively, to appear twice as a factor in $u$, a contradiction.

% the case n=4, k=5
% 1  2  3  4  5  6  7  8  9  10 11 12  13 14 15
% u1 u2 u3 u4 *  u1 u2 u3 cu4      cu1 u2 u3 u4
% 0              0                 1
% 0  1  0  0     0  1  0  1        1   1  0  0  --> 0101 twice
% 0  0  1  0     0  0  1           1   0  1  0  --> 1010 twice
% 0  1  0  1     0  1  0  0  0  0  1   1  0  1  --> 1101 twice

The proof of the third case proceeds very similarly to the second case, and allows us to conclude that such \an{} \upword{} $u$ must have the form $u=u_1u_2u_3u_4u_5u_6\Diamond u_1u_2u_3\ol{u_4}\ol{u_3}u_4u_5u_6$.
We assume w.l.o.g.\ that $u_3=0$.
The word $z:=0000$ must appear somewhere as a factor in $u$, and since $u_{12}=\ol{u_3}=1$ the only possible starting positions for $z$ in $u$ are $1,2,\ldots,8$.
The starting positions $1,3,4,6,8$ can be excluded immediately, as they would cause $z$ to appear twice as a factor in $u$.
On the other hand, if $z$ starts at positions 2, 5 or 7, then the neighboring letters must both be 1, causing 0011, 0101 or 0000, respectively, to appear twice as a factor in $u$, a contradiction.

% the case n=4, k=7
% 1  2  3  4  5  6  7  8  9  10 11  12   13  14  15
% u1 u2 u3 u4 u5 u6 *  u1 u2 u3 cu4 cu3  u4  u5  u6
%       0                    0      1
% 1  0  0  0  0  1     1  0  0  1   1        0   1  --> 0011 twice
% 0  1  0  1  0  0     0  1  0  0   1        0   0  --> 0101 twice
% 0  0  0  0     1     0  0  0  1   1    0          --> 0000 twice
\end{proof}

\subsection{Existence results}

We conjecture that for a binary alphabet and a single $\Diamond$, the non-existence cases discussed in the previous section are the only ones.

\begin{conjecture}
\label{conj:single-diamond}
For $A=\{0,1\}$ and any $n\geq 1$, there is a linear \upword{} for $A^n$ containing a single $\Diamond$ at position $k$ in every case not covered by Theorem~\ref{thm:diamond-at-pos-n} or Theorem~\ref{thm:diamond-at-pos-57}.
\end{conjecture}

Recall the numerical evidence for the conjecture discussed in the introduction.
In the remainder of this section we prove some cases of this general conjecture.

\begin{theorem}
\label{thm:diamond-at-pos-1}
For $A=\{0,1\}$ and any $n\geq 2$, there is a linear \upword{} for $A^n$ with a single $\Diamond$ at the first position that begins with $\Diamond 0^{n-1}1$.
\end{theorem}

Note that by Lemma~\ref{lem:constraint}, {\em every} linear \upword{} for $A^n$ with a single $\Diamond$ of the form $u=\Diamond u_2u_3\cdots u_N$ satisfies the conditions $u_2=u_3=\cdots=u_n=\ol{u_{n+1}}$, i.e., w.l.o.g.\ it begins with $\Diamond 0^{n-1}1$ (up to letter permutations).

\begin{proof}
Consider the word $v=v_1v_2\cdots v_{n+1}:=\Diamond 0^{n-1}1$ and the corresponding three edges $(0^{n-1},0^{n-1})$, $(10^{n-2},0^{n-1})$ and $(0^{n-1},0^{n-2}1)$ in the De Bruijn graph $G_A^{n-1}$.
Denote the graph obtained from $G_A^{n-1}$ by removing these three edges and the isolated vertex $0^{n-1}$ by $G'$.
Clearly, the edges of $G'$ form a connected graph, and every vertex in $G'$ has in- and out-degree exactly two, except the vertex $y:=0^{n-2}1$ which has one more out-edge than in-edges and the vertex $z:=10^{n-2}$ which has one more in-edge than out-edges.
Therefore, $G'$ has an Eulerian path starting at $y$ and ending at $z$, and this Eulerian path yields the desired \upword{} that begins with $v$.
\end{proof}

For any binary word $w\in A^k$, $A=\{0,1\}$, and any $n\geq 1$, we write $c(w,n)=c_1c_2\cdots c_n$ for the word given by $c_i=w_i$ for $i=1,2,\ldots,k$, $c_i=c_{i-k}$ for all $i=k+1,k+2,\ldots,n-1$ and $c_n=\ol{c_{n-k}}$.
Informally speaking, $c(w,n)$ is obtained by concatenating infinitely many copies of $w$, truncating the resulting word at length $n$ and complementing the last symbol.
For example, we have $c(011,7)=0110111$ and $c(011,8)=01101100$.
Using this terminology, the starting segment of the linear \upword{} from Theorem~\ref{thm:diamond-at-pos-1} can be written as $\Diamond c(0,n)$.
The next result is a considerable extension of the previous theorem.

\begin{theorem}
\label{thm:diamond-at-pos-k}
For $A=\{0,1\}$, any $n\geq 3$ and any $k\in\{2,3,\ldots,n-1\}$, there is a linear \upword{} for $A^n$ with a single $\Diamond$ at the $k$-th position that begins with $01^{k-2}\Diamond c(01^{k-1},n)$.
\end{theorem}

The idea of the proof of Theorem~\ref{thm:diamond-at-pos-k} is a straightforward generalization of the approach we used to prove Theorem~\ref{thm:diamond-at-pos-1} before, and boils down to showing that the De Bruijn graph $G_A^{n-1}$ without the edges that are given by the prescribed \upword{} prefix still has an Eulerian path.

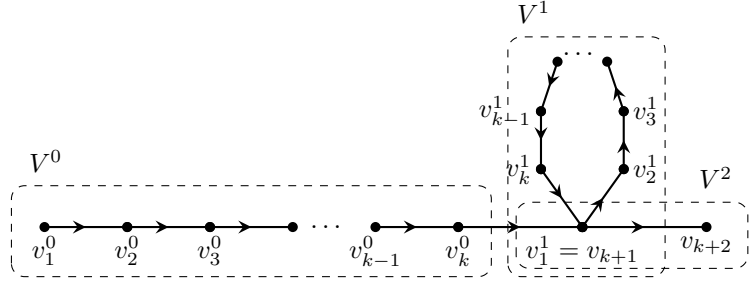
\begin{figure}
\begin{center}
\begin{tikzpicture}[scale=1.1]
\draw[rounded corners=5pt,dashed] (-0.4,-0.6) rectangle ++(5.8,1.1);
\draw[rounded corners=5pt,dashed] (5.6,-0.6) rectangle ++(1.9,2.9);
\draw[rounded corners=5pt,dashed] (5.7,-0.5) rectangle ++(2.8,0.8);

\draw [directed,thick](0,0)--++(1,0);
\draw [directed,thick](1,0)--++(1,0);
\draw [directed,thick](2,0)--++(1,0);
\draw [directed,thick](4,0)--++(1,0);
\draw [directed,thick](5,0)--++(1.5,0);
\draw [directed,thick](6.5,0)--++(1.5,0);

\draw [directed,thick](6.5,0)--++(0.5,0.7);
\draw [directed,thick](7,0.7)--++(0,0.7);
\draw [directed,thick](7,1.4)--++(-0.2,0.6);
\draw [directed,thick](6.2,2)--++(-0.2,-0.6);
\draw [directed,thick](6,1.4)--++(0,-0.7);
\draw [directed,thick](6,0.7)--++(0.5,-0.7);

\node[below] at (0,0) {$v_1^0$};
\node[below] at (1,0) {$v_2^0$};
\node[below] at (2,0) {$v_3^0$};
\node[below] at (4,0) {$v_{k-1}^0$};
\node[below] at (5,0) {$v_k^0$};
\node[below] at (6.5,0) {$v_1^1=v_{k+1}$};
\node[below] at (8,0) {$v_{k+2}$};
\node[right] at (7,0.7) {$v_2^1$};
\node[right] at (7,1.4) {$v_3^1$};
\node[left] at (6,1.4) {$v_{k-1}^1$};
\node[left] at (6,0.7) {$v_k^1$};
\node[right] at (3.1,0) {$\ldots$};
\node[right] at (6.15,2.1) {$\ldots$};
\node[right] at (-0.3,0.8) {$V^0$};
\node[right] at (5.6,2.6) {$V^1$};
\node[right] at (7.8,0.6) {$V^2$};
\end{tikzpicture} \\[1ex]
\caption{Subgraph of $G_A^{n-1}$ constructed in the proof of Theorem~\ref{thm:diamond-at-pos-k}.}
\label{fig:diamond-at-pos-k}
\end{center}
\end{figure}

\begin{proof}
The words $0\Diamond c(01,3)100=0\Diamond 011100$, $0\Diamond c(01,4)11011110000=0\Diamond 010011011110000$ and $01\Diamond c(011,4)100001010=01\Diamond 0111100001010$ from Table~\ref{tab:upwords1} show that the statement is true for $n=3$ and $n=4$.
For the rest of the proof we assume that $n\geq 5$.
Consider the word $w=w_1w_2\cdots w_{k+n}:=01^{k-2}\Diamond c(01^{k-1},n)$.
For $i=1,2,\ldots,k$ we let $v_i^0$ and $v_i^1$ denote the two words from $S(w,i,n-1)$ obtained by substituting $\Diamond$ in $w$ by $0$ or $1$, respectively.
Moreover, let $v_{k+1}=w_{k+1}\cdots w_{k+n-1}$ be the unique word from $S(w,k+1,n-1)$ and $v_{k+2}=w_{k+2}\cdots w_{k+n}$ the unique word from $S(w,k+2,n-1)$, and define $V^0:=\{v_i^0 \mid i=1,2,\ldots,k\}$, $V^1:=\{v_i^1 \mid i=1,2,\ldots,k\}$, $V^2:=\{v_{k+1},v_{k+2}\}$ and $V':=V^0\cup V^1\cup V^2$.
We proceed to show that $|V'|=2k+1$, i.e., only two of the words just defined coincide, namely $v_1^1=v_{k+1}$ ($v_1^1$ is given by the first $n-1$ letters of $w=01^{k-1}c(01^{k-1},n)$, and $v_{k+1}$ is given by the first $n-1$ letters of $c(01^{k-1},n)$, which are equal).
In other words, the corresponding set of vertices in $G_A^{n-1}$ has size $2k+1$ (see Figure~\ref{fig:diamond-at-pos-k}).
If $k=2$, then this can be verified directly by considering the number of leading and trailing 0s and 1s of the vertices in $V^0$, $V^1$ and $V^2$.
We now assume that $k\geq 3$.
Every word from $V^0$, except possibly $v_1^0$, contains the factor 00 exactly once and is uniquely identified by the position of this factor, proving that $|V^0|=k$.
The words in $V^1$ are all uniquely identified by the number of leading 1s, which equals 0 for $v_1^1$ and $k-i+1$ for $i=2,3,\ldots,k$, implying that $|V^1|=k$.
We now show that $V^0$ and $V^1$ are disjoint.
To prove this we use again that all the words in $V^0$, except possibly $v_1^0$, contain the factor 00, and that moreover no word from $V^1$ contains this factor.
However $v_1^0$ does not contain the factor 00 only in the case $k=n-1$, and then $v_1^0$ starts and ends with 0, unlike any of the words from $V^1$ in this case, proving that $V^0$ and $V^1$ are disjoint.
It remains to show that $v_{k+2}\notin V^0\cup V^1$.
If $k=n-1$, then $v_{k+2}=1^{n-1}$ and all other words from $V^0$ and $V^1$ contain at least one 0, so $v_{k+2}\notin V^0\cup V^1$.
If $k\leq n-2$, then the word $v_{k+2}=w_{k+2}\cdots w_{k+n}$ satisfies $w_{k+n}=\ol{w_n}$, i.e., its last letter and the one $k$ positions to the left of it are complementary (recall the definition of $c(01^{k-1},n)$), a property that does not hold for any of the words in $V^1$, implying that $v_{k+2}\notin V^1$.
Moreover, in this case all words from $V^0$ contain the factor 00 exactly once and are uniquely identified by the position of this factor, and $v_{k+2}$ might contain the factor 00 only at the last two positions, so the only potential conflict could arise in the case $k=n-2$ when $v_1^0=01^{n-4}00$ ends with 00.
However, in this case $v_{k+2}=1^{n-3}00$ is still different from $v_1^0$.
We conclude that $v_{k+2}\notin V^0\cup V^1$ in all cases.
Combining these observations shows that $|V'|=|V^0|+|V^1|+|V^2|-1=2k+1$, as claimed.

Consider the set of $2k+1$ edges $E':=\{(v_i^0,v_{i+1}^0)\mid i=1,2,\ldots,k-1\}\cup \{(v_i^1,v_{i+1}^1)\mid i=1,2,\ldots,k-1\}\cup\{(v_k^0,v_{k+1}),(v_k^1,v_{k+1}),(v_{k+1},v_{k+2})\}$ in the De Bruijn graph $G_A^{n-1}$ (see Figure~\ref{fig:diamond-at-pos-k}).
They span a subgraph on $V'$ that has the following pairs of out-degrees and in-degrees: $(1,0)$ for the vertex $v_1^0$, $(0,1)$ for the vertex $v_{k+2}$, $(1,1)$ for the vertices $v_i^0$ and $v_i^1$, $i=2,3,\ldots,k$, $(2,2)$ for the vertex $v_1^1=v_{k+1}$.

We denote the graph obtained from $G_A^{n-1}$ by removing the edges in $E'$ and the isolated vertex $v_1^1=v_{k+1}$ by $G'$.
Clearly, every vertex in $G'$ has the same in- and out-degree (1 or 2), except the vertex $v_{k+2}$ which has one more out-edge than in-edges, and the vertex $v_1^0$ which has one more in-edge than out-edges.
To complete the proof of the theorem we show that $G'$ contains an Eulerian path (which must start at $v_{k+2}$ and end at $v_1^0$), and to do this, it suffices (by the before-mentioned degree conditions) to show that $G'$ is connected.

We first consider the case $k\leq n-2$:
From any vertex $v\in G'$, we follow 0-edges until we either reach the vertex $0^{n-1}$ or a vertex from $V'$ for which the next 0-edge is from $E'$ (this could happen right at the beginning if $v\in V'$).
In this case we follow 1-edges until we reach the vertex $1^{n-1}$, and from there we follow 0-edges until we reach $0^{n-1}$.
(We only ever follow edges in forward direction.)
We claim that in this process we never use an edge from $E'$, which shows that $G'$ is connected.
To see this suppose we encounter a vertex $v'$ from $V'$ for which the next 0-edge is from $E'$.
This means that $v'$ has $k-1$ trailing 1s (here we use that $k\leq n-2$), so following a 1-edge leads to a vertex that has $k$ trailing 1s, and in the next step to a vertex that has $k+1$ trailing 1s.
Note that the vertices in $V'\setminus\{v_{k+2}\}$ have at most $k-1$ trailing 1s, and $v_{k+2}$ has at most $k$ trailing 1s, so we avoid any edges from $E'$ on our way to $1^{n-1}$.
Moreover, on the way from $1^{n-1}$ to $0^{n-1}$ via $1^{n-1-i}0^i$, $i=1,2,\ldots,n-1$, we do not use any edges from $E'$ either, because any vertex from $V'\setminus\{v_{k+2}\}$ that starts with a 1 has at least two transitions from 1s to 0s, or vice versa, when reading it from left to right (using again $k\leq n-2$), and $0^{n-1}\notin V'$.

Now consider the case $k=n-1$:
From any vertex $v\in G'$, we follow 0-edges until we either reach the vertex $0^{n-1}$ or the only vertex $v_1^0=01^{n-3}0$ from $V'\setminus \{v_{k+1}\}$ for which the next 0-edge is from $E'$.
In this case we follow a single 1-edge to $1^{n-3}01=v_3^1$, and from there we follow 0-edges until we reach $0^{n-1}$.
Similarly to before, we need to argue that we never use an edge from $E'$ in this process.
On the way from $1^{n-3}01=v_3^1$ to $0^{n-1}$ we never use any edges from $E'$, because any vertices on this path except the first one $1^{n-3}01$ and the last two $10^{n-2}$ and $0^{n-1}$ contain the factor 010, so all these vertices are different from $V'$ (for $n\geq 5$ and $k=n-1$ no word from $V'$ contains 010 as a factor), implying that all edges except possibly the last one are safe.
However, since $0^{n-1}\notin V'$, the last edge is safe, too.

These arguments show that $G'$ is connected, so it has an Eulerian path, and this Eulerian path yields the desired \upword{} that begins with $w$.
This completes the proof.
\end{proof}

\section{Linear \upwords{} with two \texorpdfstring{$\Diamond$s}{diamonds}}
\label{sec:two-diamonds}

In this section we focus on binary alphabets.
Many of the non-existence conditions provided in this section can be generalized straightforwardly to non-binary alphabets, as we briefly discuss in Section~\ref{sec:outlook} below.

\subsection{Non-existence results}

\begin{theorem}
\label{thm:two-diamonds-non}
For $A=\{0,1\}$ and any $n\geq 5$, there is no linear \upword{} for $A^n$ with two $\Diamond$s of the form $u=x\Diamond y\Diamond z$ if $|x|,|y|,|z|\geq n$ or $|x|=n-1$ or $|z|=n-1$ or $|y|\leq n-2$.
\end{theorem}

As Table~\ref{tab:upwords2} shows, there are examples of linear \upwords{} with two $\Diamond$s whenever the conditions in Theorem~\ref{thm:two-diamonds-non} are violated.
Put differently, for every \upword{} $u=x\Diamond y\Diamond z$ in the table for $n\geq 5$ we have that one of the numbers $|x|,|y|,|z|$ is at most $n-1$, $|x|\neq n-1$, $|z|\neq n-1$ and $|y|\geq n-1$.
Note that already by the first condition $|x|,|y|,|z|\geq n$, a $(1-o(1))$-fraction of all choices of placing two $\Diamond$s among $N=\Theta(2^n)$ positions are excluded as possible candidates for \upwords{}.

\begin{proof}
We first assume that $|x|,|y|,|z|\geq n$, i.e., $y_n,z_n\in A$.
Applying Lemma~\ref{lem:constraint} yields $z_i=y_i=x_i\in A$ for $i=1,2,\ldots,n-1$ and $z_n=y_n=\ol{x_n}$, so the word $y_1y_2\cdots y_n=z_1z_2\cdots z_n$ appears twice as a factor in $u$, a contradiction.

We now assume that $|x|=n-1$ (the case $|z|=n-1$ follows by symmetry).
Note that the number of factors of $u$ is at most $2(|y|+1)+4(|z|+1)$:
This is because every subword ending at the first $\Diamond$ or at a letter from $y$ contains at most one $\Diamond$, giving rise to two factors, and every subword ending at the second $\Diamond$ or at a letter from $z$ contains at most two $\Diamond$s, giving rise to four factors.
This number is at most $2n+4n=6n$ for $|y|,|z|\leq n-1$, which is strictly less than $2^n$ for $n\geq 5$.
Therefore, we must have $|y|\geq n$ or $|z|\geq n$ in this case.
We assume w.l.o.g.\ that $|y|\geq n$, i.e., $y_n\in A$.
Applying Lemma~\ref{lem:constraint} yields $y_i=x_i\in A$ for $i=1,2,\ldots,n-1$, implying that the word $y_1y_2\cdots y_n$ appears twice as a factor in $u$, a contradiction.

We now assume that $|y|\leq n-2$.
In this case we must have $|x|\geq n$ or $|z|\geq n$, because if $|x|,|z|\leq n-1$ then the number of factors of $u$ is at most $2(|y|+1)+4(|z|+1)\leq 2(n-1)+4n\leq 6n$, which is strictly less than $2^n$ for $n\geq 5$.
We assume w.l.o.g.\ that $|z|\geq n$.
Let $k:=|y|+1\leq n-1$ and consider the subword $y':=y\Diamond z_1z_2\cdots z_{n-k}$ of $u$, which is well-defined since $|z|\geq n$ ($k$ is the position of the $\Diamond$ in $y'$).
Since $k\leq n-1$ we have $y'_n=z_{n-k}\in A$.
Applying Lemma~\ref{lem:constraint} yields that $|x|=|y|$.
Moreover, if $k=1$ ($|x|=|y|=0$) then the same lemma yields $y'_2=y'_3=\cdots =y'_{n-1}=\ol{y'_n}$, i.e., $z_1=z_2=\cdots=z_{n-2}=\ol{z_{n-1}}$ and $z_{n-1}=z_{n-3}$, a contradiction.
On the other hand, if $k\geq 2$, then $z_{i+k\ell}=y_i=x_i$ for all $i=1,2,\ldots,k-1$ and $\ell=0,1,\ldots$ with $i+k\ell\leq n-1$, i.e., the factors obtained from the subword $y'$ in $u$ appear twice, starting at position 1 and position $k+1$, a contradiction.
\end{proof}

\begin{corollary}
\label{cor:dia-dia}
For $A=\{0,1\}$ and any $n\geq 2$, $\Diamond\Diamond$ for $n=2$ and $\Diamond\Diamond 0111$ for $n=3$ are the only linear \upwords{} for $A^n$ containing two $\Diamond$s that are adjacent (up to reversal and letter permutations).
\end{corollary}

\begin{proof}
The non-existence of linear \upwords{} with two adjacent $\Diamond$s for $n\geq 5$ follows from Theorem~\ref{thm:two-diamonds-non}, because for such \an{} \upword{} $u=x\Diamond\Diamond z$ the subword $y$ between the two $\Diamond$s is empty, so $|y|=0\leq n-2$.
For $n=4$ and $|y|=0$ the estimate in the third part in the proof of Theorem~\ref{thm:two-diamonds-non} can be strengthened to show that if $|x|,|z|\leq n-1$, then the number of factors of $u$ is strictly less than $4n\leq 2^n$ unless $u=u_1u_2u_3\Diamond\Diamond u_6u_7u_8$, which means we can continue the argument as before, leading to a contradiction.
The exceptional case $u=u_1u_2u_3\Diamond\Diamond u_6u_7u_8$ can be excluded as follows: Applying Lemma~\ref{lem:constraint} shows that $u_2=u_6$ and $u_3=u_7$, and then it becomes clear that the factor $0000$, at whatever position within $u$ it is placed, would appear twice.
For $n=3$ the only possible linear \upwords{} with two adjacent $\Diamond$s by Lemma~\ref{lem:constraint} are $u=\Diamond\Diamond u_3\ol{u_3}\ol{u_3}u_6$, which leads to $\Diamond\Diamond 0111$ (w.l.o.g.\ $u_3=0$, and for 111 to be covered we must have $u_6=1$), and $u=u_1\Diamond\Diamond u_4$ is impossible because $u_10u_4$ appears twice as a factor (starting at positions 1 and 2).
For $n=2$ the only possible linear \upword{} with two $\Diamond$s is $\Diamond\Diamond$.
\end{proof}

\subsection{Existence results}

Our next result provides an infinite number of binary linear \upwords{} with two $\Diamond$s (see Table~\ref{tab:upwords2}).

\begin{theorem}
\label{thm:two-diamonds}
For $A=\{0,1\}$ and any $n\geq 4$, there is a linear \upword{} for $A^n$ with two $\Diamond$s that begins with $\Diamond 0^{n-1}1^{n-2}\Diamond 1 0^{n-2} 1$.
\end{theorem}

\begin{proof}
Consider the word $w=w_1w_2\cdots w_{3n-1}:=\Diamond 0^{n-1}1^{n-2}\Diamond 1 0^{n-2} 1$.
It is easy to check that $w$ yields $3n+1$ different factors $x_1x_2\cdots x_n\in A^n$, and each of these factors gives rise to an edge $(x_1x_2\cdots x_{n-1},x_2x_3\cdots x_n)$ in the De Bruijn graph $G_A^{n-1}$.
The set $E'$ of these edges and their end vertices $V'$ form a connected subgraph that has in- and out-degree 1 for all vertices in $V'$ except for $v_0':=0^{n-1}$, $v_1':=1^{n-1}$, $v_2':=10^{n-2}$ and $v_3':=1^{n-2}0$ which have in- and out-degree 2, and $y:=0^{n-2}1$ and $z:=01^{n-2}$ which have in-degree 2 and out-degree 1, or in-degree 1 and out-degree 2, respectively.
We denote the graph obtained from $G_A^{n-1}$ by removing the edges in $E'$ and the vertices $v_0',v_1',v_2'$ and $v_3'$ by $G'$.
Clearly, every vertex in $G'$ has the same in- and out-degree (1 or 2), except the vertex $y$ which has only one outgoing edge, and the vertex $z$ which has only one incoming edge.
To complete the proof of the theorem we show that $G'$ contains an Eulerian path (which must start at $y$ and end at $z$), and to do this, it suffices (by the before-mentioned degree conditions) to show that $G'$ is connected.

If $n=4$, then $G'$ consists only of the edges $(y,010),(010,101),(101,z)$ (a connected graph), so for the rest of the proof we assume that $n\geq 5$.
Consider a vertex $v$ in $G'$ other than $z$.

If $v$ ends with 0, consider the (maximum) number $k$ of trailing 0s.
Note that $k\leq n-3$, as the vertices $v_2'$ and $v_0'$ that correspond to the cases $k\in\{n-2,n-1\}$ are not in $G'$.
From $v$ we follow 1-edges and 0-edges alternatingly, starting with a 1-edge, until we either reach the vertex $s:=1^{n-3}01$ or the vertex $t:=010101\cdots\in A^{n-1}$ (this could happen right at the beginning if $v=t$).
From $s$ or $t$ we follow 1-edges until the vertex $z$.

If $v$ ends with 1, then we do the following: If $v\neq s$ we follow a single 0-edge, and then proceed as before until the vertex $z$.
If $v=s$ we directly follow 1-edges until $z$.
(Note that we only ever follow edges in forward direction.)

We claim that in this process we never use an edge from $E'$, which shows that $G'$ is connected.
To see this we first consider the case that we start at a vertex $v$ with $k\leq n-3$ trailing 0s.
If $k\geq 2$, then the vertex reached from $v$ via a 1-edge is not in $V'$, because no vertex in $V'$ has a segment of $k\leq n-3$ consecutive 0s surrounded by 1s.
Also, none of the next vertices before reaching $t$ is from $V'$, because all contain the factor $0010$, unlike any word in $V'$.
If $k=1$, then the vertex reached from $v$ by following a 1-edge is either $s\in V'$ (then we stop) or not in $V'$, as no other vertex from $V'$ ends with 101.
If it is not in $V'$, then the next vertex reached via a 0-edge could be in $V'$, but all the subsequent vertices until (and including) $t$ are not, since they all contain the factor 0101, unlike any word in $V'$.
This shows that none of the edges traversed from $v$ to $s$ or $t$ is from $E'$.
Moreover, none of the vertices traversed between $s$ and $z$ or between $t$ and $z$ is from $V'$, because they all contain the factor $0101$ or $1011$, unlike any word in $V'$, so we indeed reach $z$ without using any edges from $E'$.

Now consider the case that we start at a vertex $v$ that ends with 1.
The only interesting case is $v\neq s$.
There are only two 0-edges in $E'$ starting at a vertex that ends with 1, namely the edges starting at $v_1'$ and $s$.
However, $v$ is different from $v_1'$ because $v_1'$ is not part of $G'$, and $v$ is different from $s$ by assumption.
We conclude that the 1-edge we follow is not from $E'$.

These arguments show that $G'$ is connected, so it has an Eulerian path, and this Eulerian path yields the desired \upword{} that begins with $w$.
This completes the proof.
\end{proof}

\section{Cyclic \upwords{}}
\label{sec:cyclic}

Throughout this section, all indices are considered modulo the size of the corresponding word.
All the notions introduced in Section~\ref{sec:prelim} can be extended straightforwardly to cyclic \upwords{}, where factors are taken cyclically across the word boundaries.
In particular, when defining the graph $H(u,n)$ for some cyclic \upword{} $=u_1u_2\cdots u_N$ we consider the subsets of words $S(u,i,n)$ cyclically for all $i=1,2,\ldots,N$.
Then the first two statements of Observation~\ref{obs:Hu-degrees} hold for all vertices $S(u,i,n)$, $i=1,2,\ldots,N$.
The next lemma is the analogue of Lemma~\ref{lem:constraint} for cyclic \upwords{}.

\begin{lemma}
\label{lem:constraint-cyclic}
Let $u=u_1u_2\cdots u_N$ be a cyclic \upword{} for $A^n$, where $A=\{0,1,\ldots,\alpha-1\}$ and $n\geq 2$.
If $u_k=\Diamond$ then $u_{k+n}=\Diamond$.
\end{lemma}

\begin{proof}
Suppose that $u_k=\Diamond$ and $u_{k+n}\neq \Diamond$.
By Observation~\ref{obs:Hu-degrees}, each vertex in the set $S(u,k+1,n)$ has in-degree $\alpha$ in $H(u,n)$, and each vertex in $S(u,k,n)$ has out-degree 1.
By Observation~\ref{obs:common-neighbors}, for each $v=v_1v_2\cdots v_n\in S(u,k+1,n)$ there are $\alpha-1$ other vertices (different from the ones in $S(u,k+1,n)$) in $G_A^n$ with the same set $\Gamma^-(v)$ of $\alpha$ many in-neighbors, namely $v_x:=v_1\cdots v_{n-1}x$, where $x\in A\setminus \{v_n\}$.
As the in-degree of every vertex of $G_A^n$ is exactly $\alpha$, and in $H(u,n)$ all vertices already have in-degree at least 1, it follows that the vertices $v_x$ can not be part of $H(u,n)$, a contradiction to the fact that $H(u,n)$ is a spanning subgraph of $G_A^n$.
\end{proof}

Lemma~\ref{cor:cyclic} immediately yields the following corollary, which captures various rather severe conditions that a cyclic \upword{} must satisfy, relating its length $N$, the size $\alpha$ of the alphabet, and the value of the parameter $n$.

\begin{corollary}
\label{cor:cyclic}
Let $u=u_1u_2\cdots u_N$ be a cyclic \upword{} for $A^n$, where $A=\{0,1,\ldots,\alpha-1\}$ and $n\geq 2$, with at least one $\Diamond$.
Then we have $N=\alpha^{n-d}$ for some $d$, $1\leq d\leq n-1$, such that $n$ divides $dN$.
\end{corollary}

\begin{proof}
By Lemma~\ref{lem:constraint-cyclic}, for any $\Diamond$ in $u$, the other two symbols in distance $n$ from it must be $\Diamond$s as well.
Thus, the indices $1,2,\ldots,N$ are partitioned into $\gcd(n,N)$ many residue classes modulo $n$, and all symbols at positions from the same residue class are either all $\Diamond$s or all letters from $A$.
Let $d$ denote the number of $\Diamond$s among any $n$ consecutive symbols of $u$, then we have $1\leq d\leq n-1$ (there is at least one $\Diamond$, but not all letters can be $\Diamond$s), and any starting position in $u$ gives rise to $\alpha^d$ different factors, implying that $N=\alpha^{n-d}$.
Furthermore, the $d$ many $\Diamond$s within any $n$ consecutive letters of $u$ are partitioned into $n/\gcd(n,N)$ many blocks with the same $\Diamond$ pattern, so $n/\gcd(n,N)$ must divide $d$, and this condition is equivalent to $n$ dividing $d\gcd(n,N)$ and to $n$ dividing $dN$.
\end{proof}

As an immediate corollary of our last result, we can exclude the existence of cyclic \upwords{} for many combinations of $\alpha$ and $n$.

\begin{corollary}
\label{cor:cyclic-div}
Let $A=\{0,1,\ldots,\alpha-1\}$ and $n\geq 2$.
If $\gcd(\alpha,n)=1$, then there is no cyclic \upword{} for $A^n$.
In particular, for $\alpha=2$ and odd $n$, there is no cyclic \upword{} for $A^n$.
\end{corollary}

\begin{proof}
Suppose that such \an{} \upword{} $u=u_1u_2\cdots u_N$ exists.
Then by Corollary~\ref{cor:cyclic} we have $N=\alpha^{n-d}$ for some $d$, $1\leq d\leq n-1$, such that $n$ divides $dN$.
However, as $\gcd(\alpha,n)=1$, $n$ does not divide $N=\alpha^{n-d}$, so $n$ must divide $d$, which is impossible, yielding a contradiction.
\end{proof}

By Corollaries~\ref{cor:cyclic} and \ref{cor:cyclic-div}, for a binary alphabet ($\alpha=2$), the only remaining potential parameter values for cyclic \upwords{} are $n=2$ and $d=1$, $n=4$ and $d\in\{1,2\}$, $n=6$ and $d=3$, $n=8$ and $d\in\{1,2,\ldots,6\}$, $n=10$ and $d=5$, $n=12$ and $d\in\{3,6,9\}$, etc.
The case $n=2$ and $d=1$ can be easily exluded: w.l.o.g.\ such a word has the form $\Diamond 0$, leading to the factor 00 appearing twice (and 11 does not appear as a factor at all).
However, for $n=4$ and $d=1$ we have the cyclic \upword{} $\Diamond 001\Diamond 110$, which we already mentioned in the introduction.
This is the only cyclic \upword{} for a binary alphabet that we know of.
Cyclic \upwords{} for any even alphabet size $\alpha\geq 4$ and $n=4$ have been constructed in the follow-up paper \cite{kirsch:16}.

\section{Outlook}
\label{sec:outlook}

In this paper we initiated the systematic study of universal partial words, and we hope that our results and the numerous examples of \upwords{} provided in the tables (see also the extensive data available on the website \cite{www}) generate substantial interest for other researchers to continue this exploration, possibly in one of the directions suggested below.

Concerning the binary alphabet $A=\{0,1\}$, it would be interesting to achieve complete classification of linear \upwords{} containing a single $\Diamond$, as suggested by Conjecture~\ref{conj:single-diamond}.
For two $\Diamond$s such a task seems somewhat more challenging (recall Table~\ref{tab:upwords2}, Theorem~\ref{thm:two-diamonds-non} and see the data from \cite{www}).
Some examples of binary linear \upwords{} with three $\Diamond$s are listed in Table~\ref{tab:upwords3}, and deriving some general existence and non-existence results for this setting would certainly be of interest.

\begin{table}
\begin{center}
\begin{tabular}{l|l}
$n$   & \\\hline
3     & $\Diamond \Diamond \Diamond $ \\ \hline
4     & $\Diamond \Diamond \Diamond 01111$ (Thm.~\ref{thm:nm1-diamonds}) \\
      & $\Diamond \Diamond 001\Diamond 11010$ \\
      & $\Diamond 001\Diamond 110\Diamond 00$ \\
      & $0\Diamond 001\Diamond 110\Diamond 0$ \\ \hline
5     & $\Diamond 0010\Diamond 0111\Diamond 10011011000001$ \\
      & $\Diamond 0000111\Diamond 10001001101100101\Diamond 1$ \\
      & $\Diamond 00001110\Diamond 100010100110101111\Diamond $ \\
      & $\Diamond 0000100111\Diamond 10001101100101\Diamond 1$ \\
      & $\Diamond 0000101110\Diamond 1000110101001111\Diamond $ \\
      & $\Diamond 00001111101\Diamond 10001011001\Diamond 01$ \\
      & $\Diamond 000010101110\Diamond 10001101001111\Diamond $ \\
      & $\Diamond 0000101001110\Diamond 1000110101111\Diamond $ \\
      & $\Diamond 00001101100111\Diamond 1000100101\Diamond 1$ \\
      & $\Diamond 0000110101001110\Diamond 1000101111\Diamond $ \\
      & $\Diamond 00001101100100111\Diamond 1000101\Diamond 1$ \\
      & $\Diamond 000010010101111100\Diamond 1101\Diamond 00$ \\
      & $0\Diamond 1100\Diamond 001111101101000101\Diamond 1$ \\
\end{tabular}
\caption{Examples of linear \upwords{} for $A^n$, $A=\{0,1\}$, with three $\Diamond$s for $n=3,4,5$.}
\label{tab:upwords3}
\end{center}
\end{table}

% n=5, d=4:
% ****011111
% *0010*0111*100110110000*

The next step would be to consider the situation of more than three $\Diamond$s present in a linear \upword{}.
The following easy-to-verify example in this direction was communicated to us by Rachel Kirsch \cite{kirsch:16}.

\begin{theorem}
\label{thm:nm1-diamonds}
For $A=\{0,1\}$ and any $n\geq 2$, $\Diamond^{n-1}01^n$ is a linear \upword{} for $A^n$ with $n-1$ many $\Diamond$s.
\end{theorem}

Complementing Theorem~\ref{thm:nm1-diamonds}, we can prove the following non-existence result in this direction, but it should be possible to obtain more general results.

\begin{theorem}
\label{thm:diamonds-start}
For $A=\{0,1\}$, any $n\geq 4$ and any $2\leq d\leq n-2$, there is no linear \upword{} for $A^n$ that begins with $\Diamond^d x_{d+1}x_{d+2}\ldots x_{n+2}$ with $x_i\in A$ for all $i=d+1,\ldots,n+2$.
\end{theorem}

The proof of Theorem~\ref{thm:diamonds-start} is easy by applying Lemma~\ref{lem:constraint} to the first and second $\Diamond$.
We leave the details to the reader.

It would also be interesting to find examples of binary cyclic \upwords{} other than $\Diamond 001\Diamond 110$ for $n=4$ mentioned before.

Finally, a natural direction would be to search for (linear or cyclic) \upwords{} for {\em non-binary} alphabets, but we anticipate that no non-trivial \upwords{} exist in most cases (recall Theorem~\ref{thm:alpha3}).
As evidence for this we have the following general non-existence result in this setting.

\begin{theorem}
\label{thm:alpha3p}
For $A=\{0,1,\ldots,\alpha-1\}$, $\alpha\geq 3$, and any $d\geq 2$, for large enough $n$ there is no linear or cyclic \upword{} for $A^n$ with exactly $d$ many $\Diamond$s.
\end{theorem}

Theorem~\ref{thm:alpha3p} shows in particular that for a fixed alphabet size $\alpha$ and a fixed number $d\geq 2$ of diamonds, there are only finitely many possible candidates for \upwords{} with $d$ diamonds (which in principle could all be checked by exhaustive search).
The proof idea is that for fixed $d$ and large enough $n$, such \an{} \upword{} must contain a $\Diamond$ and a symbol from $A$ in distance $n$, and then applying Lemma~\ref{lem:constraint} or Lemma~\ref{lem:constraint-cyclic} yields a contradiction (recall the proof of Theorem~\ref{thm:alpha3}).
We omit the details here.
On the positive side, \upwords{} for even alphabet sizes $\alpha\geq 4$ and $n=4$ have been constructed in \cite{kirsch:16} (and these \upwords{} are even cyclic).

A question that we have not touched in this paper is the algorithmic problem of efficiently generating \upwords{}.
As a preliminary observation in this direction we remark here that some of the linear \upwords{} constructed in Theorem~\ref{thm:diamond-at-pos-1} and \ref{thm:diamond-at-pos-k} can also be obtained by straightforward modifications of the FKM de Bruijn sequences constructed in \cite{MR523071,MR855323}, for which efficient generation algorithms are known \cite{MR1176670}.

\acknowledgements

The authors thank Martin Gerlach for his assistance in our computer searches, Rachel Kirsch and her collaborators \cite{kirsch:16}, as well as Artem Pyatkin for providing particular examples of small \upwords{}.
The second author is grateful to Sergey Avgustinovich for helpful discussions on universal partial words, and to Bill Chen and Arthur Yang for their hospitality during the author's visit of the Center for Combinatorics at Nankai University in November 2015.
This work was supported by the 973 Project, the PCSIRT Project of the Ministry of Education and the National Science Foundation of China.
We also thank the anonymous referees of this paper for several valuable suggestions and references that helped improving the presentation.

\bibliographystyle{alpha}
\bibliography{refs}

\newcommand{\etalchar}[1]{$^{#1}$}
\begin{thebibliography}{BSBK{\etalchar{+}}09}

\bibitem[BB99]{MR1687780}
J.~Berstel and L.~Boasson.
\newblock Partial words and a theorem of {F}ine and {W}ilf.
\newblock {\em Theoret. Comput. Sci.}, 218(1):135--141, 1999.
\newblock WORDS (Rouen, 1997).

\bibitem[BBSGR10]{MR2779632}
B.~Blakeley, F.~Blanchet-Sadri, J.~Gunter, and N.~Rampersad.
\newblock On the complexity of deciding avoidability of sets of partial words.
\newblock {\em Theoret. Comput. Sci.}, 411(49):4263--4271, 2010.

\bibitem[BJG08]{bang-jensen-gutin:08}
J.~Bang-Jensen and G.~Z. Gutin.
\newblock {\em Digraphs: Theory, Algorithms and Applications}.
\newblock Springer, 2nd edition, 2008.

\bibitem[BS05]{blanchet-sadri:05}
F.~Blanchet-Sadri.
\newblock Primitive partial words.
\newblock {\em Discrete Applied Mathematics}, 148(3):195--213, 2005.

\bibitem[BS08]{blanchet-sadri:08}
F.~Blanchet-Sadri.
\newblock Open problems on partial words.
\newblock In G.~Bel-Enguix, M.~D. Jim{\'e}nez-L{\'o}pez, and
  C.~Mart{\'i}n-Vide, editors, {\em New Developments in Formal Languages and
  Applications}, pages 11--58. Springer Berlin Heidelberg, Berlin, Heidelberg,
  2008.

\bibitem[BSBK{\etalchar{+}}09]{MR2511569}
F.~Blanchet-Sadri, N.~C. Brownstein, A.~Kalcic, J.~Palumbo, and T.~Weyand.
\newblock Unavoidable sets of partial words.
\newblock {\em Theory Comput. Syst.}, 45(2):381--406, 2009.

\bibitem[CDG92]{MR1197444}
F.~Chung, P.~Diaconis, and R.~Graham.
\newblock Universal cycles for combinatorial structures.
\newblock {\em Discrete Math.}, 110(1-3):43--59, 1992.

\bibitem[CJJK11]{MR2770130}
A.~Claesson, V.~Jel{\'{\i}}nek, E.~Jel{\'{\i}}nkov{\'a}, and S.~Kitaev.
\newblock Pattern avoidance in partial permutations.
\newblock {\em Electron. J. Combin.}, 18(1):Paper 25, 41, 2011.

\bibitem[CPT11]{compeau:11}
P.~E.~C. Compeau, P.~A. Pevzner, and G.~Tesler.
\newblock How to apply {D}e {B}ruijn graphs to genome assembly.
\newblock {\em Nature biotechnology}, 29(11):987--991, 2011.

\bibitem[dB46]{de-bruijn:46}
N.~G. de~Bruijn.
\newblock {A Combinatorial Problem}.
\newblock {\em Koninklijke Nederlandse Akademie v. Wetenschappen}, 49:758--764,
  1946.

\bibitem[Eul36]{euler}
L.~Euler.
\newblock Solutio problematis ad geometriam situs pertinentis.
\newblock {\em Comment. Academiae Sci. I. Petropolitanae}, 8:128--140, 1736.

\bibitem[FK86]{MR855323}
H.~Fredricksen and I.~J. Kessler.
\newblock An algorithm for generating necklaces of beads in two colors.
\newblock {\em Discrete Math.}, 61(2-3):181--188, 1986.

\bibitem[Fle83]{fleury}
M.~Fleury.
\newblock Deux problemes de geometrie de situation.
\newblock {\em Journal de mathematiques elementaires}, pages 257--261, 1883.

\bibitem[FM78]{MR523071}
H.~Fredricksen and J.~Maiorana.
\newblock Necklaces of beads in {$k$}\ colors and {$k$}-ary de {B}ruijn
  sequences.
\newblock {\em Discrete Math.}, 23(3):207--210, 1978.

\bibitem[GGH{\etalchar{+}}16]{kirsch:16}
B.~Goeckner, C.~Groothuis, C.~Hettle, B.~Kell, P.~Kirkpatrick, R.~Kirsch, and
  R.~Solava.
\newblock Universal partial words over non-binary alphabets.
\newblock {\it arXiv:1611.03928}, Nov 2016.

\bibitem[HHK08]{MR2456602}
V.~Halava, T.~Harju, and T.~K{\"a}rki.
\newblock Square-free partial words.
\newblock {\em Inform. Process. Lett.}, 108(5):290--292, 2008.

\bibitem[HHKS09]{MR2492032}
V.~Halava, T.~Harju, T.~K{\"a}rki, and P.~S{\'e}{\'e}bold.
\newblock Overlap-freeness in infinite partial words.
\newblock {\em Theoret. Comput. Sci.}, 410(8-10):943--948, 2009.

\bibitem[HRW12]{MR2925746}
A.~E. Holroyd, F.~Ruskey, and A.~Williams.
\newblock Shorthand universal cycles for permutations.
\newblock {\em Algorithmica}, 64(2):215--245, 2012.

\bibitem[Hur90]{MR2638827}
G.~H. Hurlbert.
\newblock {\em Universal cycles: {O}n beyond de {B}ruijn}.
\newblock ProQuest LLC, Ann Arbor, MI, 1990.
\newblock Thesis (Ph.D.)--Rutgers The State University of New Jersey - New
  Brunswick.

\bibitem[HW73]{MR1509807}
C.~Hierholzer and C.~Wiener.
\newblock Ueber die {M}\"oglichkeit, einen {L}inienzug ohne {W}iederholung und
  ohne {U}nterbrechung zu umfahren.
\newblock {\em Math. Ann.}, 6(1):30--32, 1873.

\bibitem[Lem71]{MR0276108}
A.~Lempel.
\newblock {$m$}-ary closed sequences.
\newblock {\em J. Combinatorial Theory Ser. A}, 10:253--258, 1971.

\bibitem[PSCF05]{pscf:05}
J.~Pag\`{e}s, J.~Salvi, C.~Collewet, and J.~Forest.
\newblock Optimised {D}e {B}ruijn patterns for one-shot shape acquisition.
\newblock {\em Image and Vision Computing}, 23(8):707 -- 720, 2005.

\bibitem[Ral82]{MR653429}
A.~Ralston.
\newblock {D}e {B}ruijn sequences---a model example of the interaction of
  discrete mathematics and computer science.
\newblock {\em Math. Mag.}, 55(3):131--143, 1982.

\bibitem[RSW92]{MR1176670}
F.~Ruskey, C.~Savage, and T.~M.~Y. Wang.
\newblock Generating necklaces.
\newblock {\em J. Algorithms}, 13(3):414--430, 1992.

\bibitem[SBSH97]{sbsh:97}
H.~Sohn, D.~L. Bricker, J.~R. Simon, and Y.~Hsieh.
\newblock Optimal sequences of trials for balancing practice and repetition
  effects.
\newblock {\em Behavior Research Methods, Instruments, {\&} Computers},
  29(4):574--581, 1997.

\bibitem[Sch01]{scheinerman:01}
E.~R. Scheinerman.
\newblock Determining planar location via complement-free {D}e {B}rujin
  sequences using discrete optical sensors.
\newblock {\em IEEE Transactions on Robotics and Automation}, 17(6):883--889,
  Dec 2001.

\bibitem[SW14]{MR3193758}
B.~Stevens and A.~Williams.
\newblock The coolest way to generate binary strings.
\newblock {\em Theory Comput. Syst.}, 54(4):551--577, 2014.

\bibitem[www]{www}
currently \url{http://www.math.tu-berlin.de/~muetze}.

\bibitem[Yoe62]{MR0142475}
M.~Yoeli.
\newblock Binary ring sequences.
\newblock {\em Amer. Math. Monthly}, 69:852--855, 1962.

\end{thebibliography}

\end{document}